\documentclass[12pt]{amsart}
\usepackage{fullpage,url,amssymb}

\DeclareFontEncoding{OT2}{}{} 

\usepackage{color}

\DeclareMathOperator{\ev}{ev}
\newcommand{\uppermu}{{\overline{\mu}}}
\newcommand{\lowermu}{{\underline{\mu}}}
\newcommand{\homog}{{\operatorname{homog}}}
\newcommand{\calQmedium}{{\mathcal Q}^{\operatorname{medium}}}
\newcommand{\calQhigh}{{\mathcal Q}^{\operatorname{high}}}

\newcommand{\Prob}{{\operatorname{Prob}}}
\newcommand{\bad}{{\operatorname{bad}}}

\newcommand{\Aff}{{\mathbb A}}

\newcommand{\F}{{\mathbb F}}

\newcommand{\PP}{{\mathbb P}}

\newcommand{\Z}{{\mathbb Z}}

\newcommand{\mm}{{\mathfrak m}}

\newcommand{\calP}{{\mathcal P}}

\newcommand{\calT}{{\mathcal T}}

\newcommand{\II}{{\mathcal I}}

\newcommand{\OO}{{\mathcal O}}

\DeclareMathOperator{\re}{Re}

\DeclareMathOperator{\Proj}{Proj}

\newcommand{\isom}{\simeq}

\newcommand{\del}{\partial}
\newcommand{\intersect}{\cap} 
\newcommand{\Union}{\bigcup} 

\newcommand{\directsum}{\oplus} 

\newtheorem{theorem}{Theorem}[section]
\newtheorem{lemma}[theorem]{Lemma}
\newtheorem{corollary}[theorem]{Corollary}

\theoremstyle{definition}

\theoremstyle{remark}

\usepackage[
	backref,
	pdfauthor={Bjorn Poonen}, 
]{hyperref}
\usepackage[alphabetic,backrefs,lite]{amsrefs} 

\begin{document}

\title[Smooth hypersurface sections]{Smooth hypersurface sections containing a given subscheme over a finite field}
\author{Bjorn Poonen}
\thanks{This article has appeared in {\em Math.\ Research Letters} {\bf 15} (2008), no.~2, 265--271.  This research was supported by NSF grant DMS-0301280.}
\address{Department of Mathematics, University of California, 
	Berkeley, CA 94720-3840, USA}
\email{poonen@math.berkeley.edu}
\urladdr{http://math.berkeley.edu/\~{}poonen}
\date{June 29, 2007}


\maketitle

\section{Introduction}\label{S:introduction}

Let $\F_q$ be a finite field of $q=p^a$ elements.
Let $X$ be a smooth quasi-projective subscheme of $\PP^n$ 
of dimension $m \ge 0$ over $\F_q$.
N.~Katz asked for a finite field analogue of the Bertini smoothness theorem,
and in particular asked 
whether one could always find a hypersurface $H$ in $\PP^n$
such that $H \intersect X$ is smooth of dimension $m-1$.
A positive answer was proved in \cite{Gabber2001}
and \cite{Poonen-bertini2004} independently.
The latter paper proved also that in a precise sense,
a positive fraction of hypersurfaces have the required property.

The classical Bertini theorem was extended 
in \cites{Bloch1970,Kleiman-Altman1979} 
to show that the hypersurface can be chosen so as to
contain a prescribed closed smooth subscheme $Z$,
provided that the condition $\dim X > 2 \dim Z$ is satisfied.
(The condition arises naturally from a dimension-counting argument.)
The goal of the current paper is to prove an analogous result
over finite fields.
In fact, our result is stronger than that of \cite{Kleiman-Altman1979}
in that we do not require $Z \subseteq X$,
but weaker in that we assume that $Z \intersect X$ be smooth.
(With a little more work and complexity, we could prove a version
for a non-smooth intersection as well, but we restrict to the smooth
case for simplicity.)
One reason for proving our result is that it is used 
by \cite{Saito-Sato2007preprint}.

Let $S=\F_q[x_0,\ldots,x_n]$ be the homogeneous coordinate ring of $\PP^n$.
Let $S_d \subseteq S$ be the $\F_q$-subspace of homogeneous polynomials
of degree $d$.
For each $f \in S_d$, let $H_f$ be the subscheme
$\Proj(S/(f)) \subseteq \PP^n$.
For the rest of this paper,
we fix a closed subscheme $Z \subseteq \PP^n$.
For $d \in \Z_{\ge 0}$,
let $I_d$ be the $\F_q$-subspace of $f \in S_d$ that vanish on $Z$.
Let $I_{\homog} = \Union_{d \ge 0} I_d$.
We want to measure the density of subsets of $I_{\homog}$,
but under the definition in \cite{Poonen-bertini2004},
the set $I_{\homog}$ itself has density $0$ whenever $\dim Z>0$;
therefore we use a new definition of density, relative to $I_{\homog}$.
Namely, we define the {\em density} of a subset $\calP \subseteq I_{\homog}$
by
\[
	\mu_Z(\calP):= \lim_{d \rightarrow \infty} 
		\frac{\#(\calP \cap I_d)}{\# I_d},
\]
if the limit exists.
For a scheme $X$ of finite type over $\F_q$, 
define the zeta function~\cite{Weil1949}
\[
	\zeta_X(s)=Z_X(q^{-s}) 
	:= \prod_{\operatorname{closed }P \in X} 
			\left(1-q^{-s \deg P} \right)^{-1}
	= \exp \left( \sum_{r=1}^\infty 
			\frac{\#X(\F_{q^r})}{r} q^{-rs} \right);
\]
the product and sum converge when $\re(s)>\dim X$.

\begin{theorem} \label{T:main}
Let $X$ be a smooth quasi-projective subscheme of $\PP^n$ 
of dimension $m \ge 0$ over $\F_q$.
Let $Z$ be a closed subscheme of $\PP^n$.
Assume that the scheme-theoretic intersection $V:=Z \intersect X$
is smooth of dimension $\ell$.
(If $V$ is empty, take $\ell=-1$.)
Define 
\[
	\calP :=\{\, f \in I_{\homog}: 
		H_f \cap X \text{ is smooth of dimension } m-1 \,\}.
\]
\begin{enumerate}
\item[(i)]
If $m>2\ell$, then
\[
	\mu_Z(\calP) = \frac{\zeta_V(m+1)}{\zeta_V(m-\ell) \; \zeta_X(m+1) }
	= \frac{1}{\zeta_V(m-\ell) \; \zeta_{X-V}(m+1)}.
\]
In this case, in particular, for $d \gg 1$,
there exists a degree-$d$ hypersurface $H$
containing $Z$ such that $H \intersect X$ is smooth
of dimension $m-1$.
\item[(ii)]
If $m \le 2 \ell$, then $\mu_Z(\calP) = 0$.
\end{enumerate}
\end{theorem}

The proof will use the closed point sieve 
introduced in \cite{Poonen-bertini2004}.
In fact, the proof is parallel to the one in that paper,
but changes are required in almost every line.

\section{Singular points of low degree}

Let $\II_Z \subseteq \OO_{\PP^n}$ be the ideal sheaf of $Z$,
so $I_d = H^0(\PP^n,\II_Z(d))$.
Tensoring the surjection
\begin{align*}
	\OO^{\directsum (n+1)} &\to \OO \\
	(f_0,\ldots,f_n) &\mapsto x_0 f_0 + \cdots + x_n f_n
\end{align*}
with $\II_Z$, twisting by $\OO(d)$, and taking global sections shows
that $S_1 I_d = I_{d+1}$ for $d \gg 1$.
Fix $c$ such that $S_1 I_d = I_{d+1}$ for all $d \ge c$.

Before proving the main result of this section
(Lemma~\ref{L:low degree}), 
we need two lemmas.

\begin{lemma}
\label{L:surjective}
Let $Y$ be a finite subscheme of $\PP^n$.
Let
\[
	\phi_d\colon I_d = H^0(\PP^n,\II_Z(d)) \to H^0(Y,\II_Z \cdot \OO_Y(d))
\]
be the map induced by the map of sheaves $\II_Z \to \II_Z \cdot \OO_Y$
on $\PP^n$.
Then $\phi_d$ is surjective for $d \ge c+\dim H^0(Y,\OO_Y)$, 
\end{lemma}

\begin{proof}
The map of sheaves $\OO_{\PP^n} \to \OO_Y$ on $\PP^n$ is surjective
so $\II_Z \to \II_Z \cdot \OO_Y$ is surjective too.
Thus $\phi_d$ is surjective for $d \gg 1$.

Enlarging $\F_q$ if necessary,
we can perform a linear change of variable to assume
$Y \subseteq \Aff^n := \{x_0 \not=0\}$.
Dehomogenization (setting $x_0=1$) identifies $S_d$
with the space $S_d'$ of polynomials in $\F_q[x_1,\ldots,x_n]$
of total degree $\le d$.
and identifies $\phi_d$ with a map 
\[
	I_d' \to B:=H^0(\PP^n,\II_Z \cdot \OO_Y).
\]
By definition of $c$, we have $S_1' I_d' = I_{d+1}'$ for $d \ge c$.
For $d \ge b$, let $B_d$ be the image of $I_d'$ in $B$,
so $S_1' B_d = B_{d+1}$ for $d \ge c$.
Since $1 \in S_1'$, we have $I_d' \subseteq I_{d+1}'$,
so 
\[
	B_c \subseteq B_{c+1} \subseteq \cdots.
\]
But $b:=\dim B < \infty$, so $B_j=B_{j+1}$ for some $j \in [c,c+b]$.
Then
\[
	B_{j+2} = S_1' B_{j+1} = S_1' B_j = B_{j+1}.
\]
Similarly $B_j=B_{j+1}=B_{j+2}=\dots$,
and these eventually equal $B$ by the previous paragraph.
Hence $\phi_d$ is surjective for $d \ge j$,
and in particular for $d \ge c+b$.
\end{proof}

\begin{lemma}
\label{L:ideal sheaf square}
Suppose $\mm \subseteq \OO_X$ is the ideal sheaf of a closed point $P \in X$.
Let $Y \subseteq X$ be the closed subscheme whose ideal sheaf 
is $\mm^2 \subseteq \OO_X$.
Then for any $d \in \Z_{\ge 0}$.
\[
	\# H^0(Y, \II_Z \cdot \OO_Y(d)) 
	= \begin{cases}
		q^{(m-\ell)\deg P}, &\text{ if $P \in V$,} \\
		q^{(m+1)\deg P}, &\text{ if $P \notin V$.}
	\end{cases}
\]
\end{lemma}

\begin{proof}
Since $Y$ is finite, we may now ignore the twisting by $\OO(d)$.
The space $H^0(Y,\OO_Y)$ has a two-step filtration whose quotients
have dimensions $1$ and $m$ over the residue field $\kappa$ of $P$.
Thus $\# H^0(Y,\OO_Y) = (\#\kappa)^{m+1} = q^{(m+1)\deg P}$.
If $P \in V$ (or equivalently $P \in Z$), then 
$H^0(Y,\OO_{Z \intersect Y})$ has a filtration whose quotients
have dimensions $1$ and $\ell$ over $\kappa$;
if $P \notin V$, then $H^0(Y,\OO_{Z \intersect Y})=0$.
Taking cohomology of
\[
	0 \to \II_Z \cdot \OO_Y \to \OO_Y  
		\to \OO_{Z \intersect Y} \to 0
\]
on the $0$-dimensional scheme $Y$ yields
\begin{align*}
	\#H^0(Y,\II_Z \cdot \OO_Y) 
	&= \frac{\# H^0(Y,\OO_Y)}{\# H^0(Y,\OO_{Z \intersect Y})} \\
	&= \begin{cases}
		q^{(m+1)\deg P}/q^{(\ell+1)\deg P}, &\text{ if $P \in V$,} \\
		q^{(m+1)\deg P}, &\text{ if $P \notin V$.}
	\end{cases}
\end{align*}
\end{proof}

If $U$ is a scheme of finite type over $\F_q$,
let $U_{<r}$ be the set of closed points of $U$ of degree $<r$.
Similarly define $U_{>r}$.

\begin{lemma}[Singularities of low degree]
\label{L:low degree}
Let notation and hypotheses be as in Theorem~\ref{T:main},
and define
\[
	\calP_r:=\{\, f \in I_{\homog}: H_f \cap X
		\text{ is smooth of dimension $m-1$ at all $P \in X_{<r}$}\,\}.
\]
Then
\[
	\mu_Z(\calP_r)= \prod_{P \in V_{<r}}
				\left(1 - q^{-(m-\ell)\deg P} \right) \cdot
			\prod_{P \in (X-V)_{<r}}
				\left(1 - q^{-(m + 1)\deg P} \right).
\]
\end{lemma}

\begin{proof}
Let $X_{<r}=\{P_1,\dots,P_s\}$.
Let $\mm_i$ be the ideal sheaf of $P_i$ on $X$.
let $Y_i$ be the closed subscheme of $X$ with
ideal sheaf $\mm_i^2 \subseteq \OO_X$,
and let $Y= \Union Y_i$.
Then $H_f \cap X$ is singular at $P_i$ (more precisely, not smooth
of dimension $m-1$ at $P_i$) if and only if
the restriction of $f$ to a section of $\OO_{Y_i}(d)$ is zero.

By Lemma~\ref{L:surjective},
$\mu_Z(\calP)$ equals the fraction of elements in $H^0(\II_Z \cdot \OO_Y(d))$
whose restriction to a section of $\OO_{Y_i}(d)$ is nonzero for every $i$.
Thus
\begin{align*}
	\mu_Z(\calP_r) 
	&= \prod_{i=1}^s \frac{\#H^0(Y_i,\II_Z \cdot \OO_{Y_i}) - 1}
	                      {\#H^0(Y_i,\II_Z \cdot \OO_{Y_i}) } \\
	&= \prod_{P \in V_{<r}}
				\left(1 - q^{-(m-\ell)\deg P} \right) \cdot
			\prod_{P \in (X-V)_{<r}}
				\left(1 - q^{-(m + 1)\deg P} \right),
\end{align*}
by Lemma~\ref{L:ideal sheaf square}.
\end{proof}

\begin{corollary}
\label{C:low degree limit}
If $m>2\ell$, then 
\[
	\lim_{r \to \infty} \mu_Z(\calP_r) 
	= \frac{\zeta_V(m+1)}{\zeta_X(m+1) \; \zeta_V(m-\ell)}.
\]
\end{corollary}

\begin{proof}
The products in Lemma~\ref{L:low degree} are the
partial products in the definition of the zeta functions.
For convergence, we need $m-\ell > \dim V = \ell$,
which is equivalent to $m>2\ell$.
\end{proof}

\begin{proof}[Proof of Theorem~\ref{T:main}(ii)]
We have $\calP \subseteq \calP_r$.
By Lemma~\ref{L:low degree},
\[
	\mu_Z(\calP_r) \le \prod_{P \in V_{<r}}
				\left(1 - q^{-(m-\ell)\deg P} \right),
\]
which tends to $0$ as $r \to \infty$ 
if $m \le 2 \ell$.
Thus $\mu_Z(\calP)=0$ in this case.
\end{proof}

{}From now on, we assume $m>2\ell$.

\section{Singular points of medium degree}

\begin{lemma}
\label{L:singular fraction}
Let $P \in X$ is a closed point of degree $e$,
where $e \le \frac{d-c}{m+1}$.
Then the fraction of $f \in I_d$ such that $H_f \cap X$
is not smooth of dimension $m-1$ at $P$
equals 
\[
\begin{cases}
		q^{-(m-\ell)e}, &\text{ if $P \in V$,} \\
		q^{-(m+1)e}, &\text{ if $P \notin V$.}
\end{cases}
\]
\end{lemma}

\begin{proof}
This follows by applying Lemma~\ref{L:surjective}
to the $Y$ in Lemma~\ref{L:ideal sheaf square},
and then applying Lemma~\ref{L:ideal sheaf square}.
\end{proof}

Define the upper and lower densities 
$\uppermu_Z(\calP)$, $\lowermu_Z(\calP)$
of a subset $\calP \subseteq I_{\homog}$ as $\mu_Z(\calP)$
was defined, but using $\limsup$ and $\liminf$ in place of $\lim$.

\begin{lemma}[Singularities of medium degree]
\label{L:medium degree}
Define
\begin{align*}
	\calQmedium_r := \Union_{d \ge 0} \{\, f \in I_d: 
	& \;\; \text{there exists $P \in X$ 
				with $r \le \deg P \le \frac{d-b}{m+1}$}\\
	& \;\; \text{such that $H_f \cap X$ is not smooth 
		of dimension $m-1$ at $P$} \,\}.
\end{align*}
Then $\lim_{r \rightarrow \infty} \uppermu_Z(\calQmedium_r)=0$.
\end{lemma}

\begin{proof}
By Lemma~\ref{L:singular fraction}, we have
\begin{align*}
	\frac{\# (\calQmedium_r \cap I_d)}{\# I_d}
	&\le \sum_{\substack{ P \in Z \\ r \le \deg P \le \frac{d-b}{m+1} }} q^{-(m-\ell)\deg P} + \sum_{\substack{ P \in X-Z \\ r \le \deg P \le \frac{d-b}{m+1} }} q^{-(m+1)\deg P} \\
	&\le \sum_{P \in Z_{\ge r}} q^{-(m-\ell)\deg P} + \sum_{P \in (X-Z)_{\ge r}} q^{-(m+1)\deg P}.
\end{align*}
Using the trivial bound that an $m$-dimensional variety 
has at most $O(q^{em})$ closed points of degree $e$,
as in the proof of \cite{Poonen-bertini2004}*{Lemma~2.4}, 
we show that each of the two sums converges to a value that is $O(q^{-r})$ 
as $r \to \infty$,
under our assumption $m>2\ell$.
\end{proof}

\section{Singular points of high degree}

\begin{lemma}
\label{L:singular fraction 2}
Let $P$ be a closed point of degree $e$ in $\PP^n-Z$.
For $d \ge c$, the fraction of $f \in I_d$ that vanish at $P$
is at most $q^{-\min(d-c,e)}$.
\end{lemma}

\begin{proof}
Equivalently, we must show that the image of $\phi_d$ 
in Lemma~\ref{L:surjective} for $Y=P$
has $\F_q$-dimension at least $\min(d-c,e)$.
The proof of Lemma~\ref{L:surjective}
shows that as $d$ runs through the integers $c,c+1,\dots$,
this dimension increases by at least $1$
until it reaches its maximum, which is $e$.
\end{proof}

\begin{lemma}[Singularities of high degree off $V$]
\label{L:high degree off V}
Define
\[
	\calQhigh_{X-V} :=\Union_{d \ge 0} \{\, f \in I_d: 
		\exists P \in (X-V)_{>\frac{d-c}{m+1}}
	\text{ such that $H_f \cap X$ is not smooth 
		of dimension $m-1$ at $P$} \,\}
\]
Then $\uppermu_Z(\calQhigh_{X-V})=0$.
\end{lemma}

\begin{proof}
It suffices to prove the lemma with $X$ replaced by each of the sets
in an open covering of $X-V$,
so we may assume $X$ is contained in $\Aff^n = \{x_0\ne 0\} \subseteq \PP^n$,
and that $V=\emptyset$.
Dehomogenize by setting $x_0=1$,
to identify $I_d \subseteq S_d$ 
with subspaces of $I_d' \subseteq S_d' \subseteq A:=\F_q[x_1,\ldots,x_n]$.

Given a closed point $x \in X$,
choose a system of local parameters $t_1,\dots,t_n \in A$
at $x$ on $\Aff^n$ such that $t_{m+1}=t_{m+2}=\dots=t_n=0$
defines $X$ locally at $x$.
Multiplying all the $t_i$ by an element of $A$ vanishing on $Z$
but nonvanishing at $x$,
we may assume in addition that all the $t_i$ vanish on $Z$.
Now $dt_1,\dots,dt_n$ are a $\OO_{\Aff^n,x}$-basis for 
the stalk $\Omega^1_{\Aff^n/\F_q,x}$.
Let $\del_1,\dots,\del_n$ be the dual basis 
of the stalk $\calT_{\Aff^n/\F_q,x}$
of the tangent sheaf.
Choose $s \in A$ with $s(x) \not=0$ to clear denominators
so that $D_i:=s \del_i$ gives a global derivation $A \rightarrow A$
for $i=1,\dots,n$.
Then there is a neighborhood $N_x$ of $x$ in $\Aff^n$
such that $N_x \cap \{t_{m+1}=t_{m+2}=\dots=t_n=0\} = N_x \intersect X$,
$\Omega^1_{N_x/\F_q}= \directsum_{i=1}^n \OO_{N_x} dt_i$,
and $s \in \OO(N_u)^*$.
We may cover $X$ with finitely many $N_x$,
so we may reduce to the case where $X \subseteq N_x$ for a single $x$.
For $f \in I_d' \isom I_d$, $H_f \cap X$ 
{\em fails} to be smooth of dimension $m-1$ at a point $P \in U$ 
if and only if $f(P)=(D_1f)(P)=\dots=(D_mf)(P)=0$.

Let $\tau=\max_i(\deg t_i)$, $\gamma=\lfloor (d-\tau)/p \rfloor$,
and $\eta = \lfloor d/p \rfloor$.
If $f_0 \in I_d'$, 
$g_1 \in S_{\gamma}'$, 
\dots, 
$g_m \in S_{\gamma}'$, 
and $h \in I_{\eta}'$
are selected uniformly and independently at random,
then the distribution of 
\[
	f:=f_0 + g_1^p t_1 + \dots + g_m^p t_m + h^p
\]
is uniform over $I_d'$, because of $f_0$.
We will bound the probability that an $f$ constructed in this way
has a point $P \in X_{>\frac{d-c}{m+1}}$
where $f(P)=(D_1f)(P)=\dots=(D_mf)(P)=0$.
We have $D_i f=(D_i f_0) + g_i^p s$ for $i=1,\dots,m$.
We will select $f_0,g_1,\dots,g_m,h$ one at a time.
For $0 \le i \le m$, define 
\[
	W_i:=X \cap \{D_1f=\dots=D_if=0\}.
\]

\bigskip
\noindent{\em Claim 1:} For $0 \le i \le m-1$,
conditioned on a choice of $f_0,g_1,\dots,g_i$
for which $\dim(W_i) \le m-i$,
the probability that $\dim(W_{i+1}) \le m-i-1$
is $1-o(1)$ as $d \rightarrow \infty$.
(The function of $d$ represented by the $o(1)$ depends on $X$ and the $D_i$.)

\bigskip
\noindent{\em Proof of Claim 1:}
This is completely analogous to the corresponding proof 
in \cite{Poonen-bertini2004}.

\bigskip
\noindent{\em Claim 2:} Conditioned on a choice of $f_0,g_1,\dots,g_m$
for which $W_m$ is finite,
$\Prob(H_f \cap W_m \cap X_{>\frac{d-c}{m+1}} = \emptyset) = 1-o(1)$
as $d \rightarrow \infty$.

\bigskip
\noindent{\em Proof of Claim 2:}
By B\'ezout's theorem as in~\cite{Fulton1984}*{p.~10},
we have $\#W_m = O(d^m)$.
For a given point $P \in W_m$,
the set $H^\bad$ of $h \in I_\eta'$ for which 
$H_f$ passes through $P$ is either $\emptyset$ 
or a coset of $\ker(\ev_P: I_\eta' \rightarrow \kappa(P))$,
where $\kappa(P)$ is the residue field of $P$,
and $\ev_P$ is the evaluation-at-$P$ map.
If moreover $\deg P> \frac{d-c}{m+1}$,
then Lemma~\ref{L:singular fraction 2} implies 
$\#H^\bad / \# I_\eta' \le q^{-\nu}$
where $\nu=\min \left(\eta,\frac{d-c}{m+1} \right)$.
Hence 
\[
	\Prob(H_f \cap W_m \cap X_{>\frac{d-c}{m+1}} \not= \emptyset) \le
		\# W_m q^{-\nu} = O(d^m q^{-\nu}) = o(1)
\]
as $d \rightarrow \infty$, since $\nu$ eventually grows linearly in $d$.
This proves Claim~2.

\bigskip
\noindent{\em End of proof:}
Choose $f \in I_d$ uniformly at random.
Claims~1 and ~2 show that
with probability $\prod_{i=0}^{m-1} (1-o(1)) \cdot (1-o(1)) = 1-o(1)$
as $d \rightarrow \infty$,
$\dim W_i = m-i$ for $i=0,1,\dots,m$
and $H_f \cap W_m \cap X_{>\frac{d-c}{m+1}} = \emptyset$.
But $H_f \cap W_m$ is the subvariety of $X$
cut out by the equations $f(P)=(D_1f)(P)=\dots=(D_mf)(P)=0$,
so $H_f \cap W_m \cap X_{>\frac{d-c}{m+1}}$ is exactly the set
of points of $H_f \cap X$ of degree $>\frac{d-c}{m+1}$
where $H_f \cap X$ is not smooth of dimension $m-1$.
Thus $\uppermu_Z(\calQhigh_{X-V})=0$.
\end{proof}

\begin{lemma}[Singularities of high degree on $V$]
\label{L:high degree on V}
Define
\[
	\calQhigh_V :=\Union_{d \ge 0} \{\, f \in I_d: 
		\exists P \in V_{>\frac{d-c}{m+1}}
	\text{ such that $H_f \cap X$ is not smooth 
		of dimension $m-1$ at $P$} \,\}.
\]
Then $\uppermu_Z(\calQhigh_V)=0$.
\end{lemma}

\begin{proof}
As before, we may assume $X \subseteq \Aff^n$ and we may dehomogenize.
Given a closed point $x \in X$,
choose a system of local parameters $t_1,\dots,t_n \in A$
at $x$ on $\Aff^n$ such that $t_{m+1}=t_{m+2}=\cdots=t_n=0$
defines $X$ locally at $x$,
and $t_1=t_2=\cdots=t_{m-\ell} = t_{m+1}=t_{m+2}=\cdots=t_n=0$
defines $V$ locally at $x$.
If $\mm_w$ is the ideal sheaf of $w$ on $\PP^n$,
then $\II_Z \to \frac{\mm_w}{\mm_w^2}$ is surjective,
so we may adjust $t_1,\ldots,t_{m-\ell}$ to assume that
they vanish not only on $V$ but also on $Z$.

Define $\del_i$ and $D_i$ as in the proof of Lemma~\ref{L:high degree off V}.
Then there is a neighborhood $N_x$ of $x$ in $\Aff^n$
such that $N_x \cap \{t_{m+1}=t_{m+2}=\dots=t_n=0\} = N_x \intersect X$,
$\Omega^1_{N_x/\F_q}= \directsum_{i=1}^n \OO_{N_x} dt_i$,
and $s \in \OO(N_u)^*$.
Again we may assume $X \subseteq N_x$ for a single $x$.
For $f \in I_d' \isom I_d$, $H_f \cap X$ 
{\em fails} to be smooth of dimension $m-1$ at a point $P \in V$ 
if and only if $f(P)=(D_1f)(P)=\dots=(D_mf)(P)=0$.

Again let $\tau=\max_i(\deg t_i)$, $\gamma=\lfloor (d-\tau)/p \rfloor$,
and $\eta = \lfloor d/p \rfloor$.
If $f_0 \in I_d'$, 
$g_1 \in S_{\gamma}'$, 
\dots, 
$g_{\ell+1} \in S_{\gamma}'$, 
are chosen uniformly at random,
then 
\[
	f:=f_0 + g_1^p t_1 + \dots + g_{\ell+1}^p t_{\ell+1}
\]
is a random element of $I_d'$, since $\ell+1 \le m-\ell$.

For $i=0,\ldots,\ell+1$, the subscheme
\[
	W_i:=V \cap \{D_1f=\dots=D_if=0\}
\]
depends only on the choices of $f_0,g_1,\ldots,g_i$.
The same argument as in the previous proof shows that 
for $i=0,\dots,\ell$, we have 
\[
	\Prob(\dim W_i \le \ell-i) = 1-o(1)
\]
as $d \to \infty$.
In particular, $W_\ell$ is finite with probability $1-o(1)$.

To prove that $\uppermu_Z(\calQhigh_V)=0$,
it remains to prove that 
conditioned on choices of $f_0,g_1,\ldots,g_\ell$
making $\dim W_\ell$ finite, 
\[
	\Prob(W_{\ell+1} \intersect V_{>\frac{d-c}{m+1}} = \emptyset) = 1-o(1).
\]
By B\'ezout's theorem, $\# W_\ell = O(d^\ell)$.
The set $H^\bad$ of choices of $g_{\ell+1}$ making 
$D_{\ell+1}f$ vanish at a given point $P \in W_\ell$
is either empty or a coset of $\ker(\ev_P: S_\gamma' \rightarrow \kappa(P))$.
Lemma~2.5 of \cite{Poonen-bertini2004} implies
that the size of this kernel (or its coset) as a fraction of $\#S_\gamma'$
is at most $q^{-\nu}$
where $\nu:=\min \left(\gamma,\frac{d-c}{m+1} \right)$.
Since $\# W_\ell q^{\nu} = o(1)$ as $d \to \infty$,
we are done.
\end{proof}

\section{Conclusion}

\begin{proof}[Proof of Theorem~\ref{T:main}(i)]
We have
\[
	\calP \subseteq \calP_r \subseteq \calP \cup \calQmedium_r \cup \calQhigh_{X-V} \cup \calQhigh_V,
\]
so
$\uppermu_Z(\calP)$ and $\lowermu_Z(\calP)$
each differ from $\mu_Z(\calP_r)$ by at most
$\uppermu_Z(\calQmedium_r) + \uppermu_Z(\calQhigh_{X-V})+\uppermu_Z(\calQhigh_{V})$.
Applying Corollary~\ref{C:low degree limit}
and Lemmas \ref{L:medium degree}, \ref{L:high degree off V}, 
and~\ref{L:high degree on V},
we obtain
\[
	\mu_Z(\calP)=\lim_{r \to \infty} \mu_Z(\calP_r) = \frac{\zeta_V(m+1)}{\zeta_V(m-\ell) \; \zeta_X(m+1) }.
\]
\end{proof}

\section*{Acknowledgements} 

I thank Shuji Saito for asking the question answered by this paper,
and for pointing out \cite{Kleiman-Altman1979}.


\begin{bibdiv}
\begin{biblist}


\bib{Bloch1970}{book}{
  author={Bloch, Spencer},
  date={1970},
  note={Ph.D.\ thesis, Columbia University},
}

\bib{Fulton1984}{book}{
  author={Fulton, William},
  title={Introduction to intersection theory in algebraic geometry},
  series={CBMS Regional Conference Series in Mathematics},
  volume={54},
  publisher={Published for the Conference Board of the Mathematical Sciences, Washington, DC},
  date={1984},
  pages={v+82},
  isbn={0-8218-0704-8},
  review={MR735435 (85j:14008)},
}

\bib{Gabber2001}{article}{
  author={Gabber, O.},
  title={On space filling curves and Albanese varieties},
  journal={Geom. Funct. Anal.},
  volume={11},
  date={2001},
  number={6},
  pages={1192\ndash 1200},
  issn={1016-443X},
  review={MR1878318 (2003g:14034)},
}

\bib{Kleiman-Altman1979}{article}{
  author={Kleiman, Steven L.},
  author={Altman, Allen B.},
  title={Bertini theorems for hypersurface sections containing a subscheme},
  journal={Comm. Algebra},
  volume={7},
  date={1979},
  number={8},
  pages={775\ndash 790},
  issn={0092-7872},
  review={MR529493 (81i:14007)},
}

\bib{Poonen-bertini2004}{article}{
  author={Poonen, Bjorn},
  title={Bertini theorems over finite fields},
  journal={Ann. of Math. (2)},
  volume={160},
  date={2004},
  number={3},
  pages={1099--1127},
  issn={0003-486X},
  review={\MR {2144974 (2006a:14035)}},
}

\bib{Saito-Sato2007preprint}{article}{
  author={Saito, Shuji},
  author={Sato, Kanetomo},
  title={Finiteness theorem on zero-cycles over $p$-adic fields},
  date={2007-04-11},
  note={{\tt arXiv:math.AG/0605165}},
}

\bib{Weil1949}{article}{
  author={Weil, Andr{\'e}},
  title={Numbers of solutions of equations in finite fields},
  journal={Bull. Amer. Math. Soc.},
  volume={55},
  date={1949},
  pages={497\ndash 508},
  review={MR0029393 (10,592e)},
}

\end{biblist}
\end{bibdiv}
\end{document}